\documentclass[11pt]{ip-journal}
\usepackage{etex}
\usepackage{latexsym}
\usepackage{bbm}
\usepackage{amscd}
\usepackage{amsmath}
\usepackage{amsfonts}
\usepackage{amssymb}

\newtheorem{theorem}{\bf Theorem}[section]

\newtheorem{prop}[theorem]{\bf Proposition}

\newtheorem{remark}[theorem]{\bf Remark}

\renewcommand{\theequation}{\arabic{section}.\arabic{equation}}
\renewcommand{\thefigure}{\thesection.\arabic{figure}}
\renewcommand{\thetable}{\thesection.\arabic{table}}

\def\no{\noindent}

\makeatletter

\newcommand{\Rmnum}[1]{\expandafter\@slowromancap\romannumeral #1@}
\makeatother
\numberwithin{equation}{section}
\usepackage{enumitem}
\usepackage{tikz-cd}
\usepackage{indentfirst}
\begin{document}
\bibliographystyle{alpha}
\begin{center}{\LARGE \bf Area-minimizing Cones over Products of Grassmannian Manifolds}
\end{center}

\begin{center}
Xiaoxiang Jiao, Hongbin Cui$^\ast$ and Jialin Xin

School of Mathematical Sciences, University of Chinese Academy of Sciences, Beijing 100049, China,

$^\ast$Corresponding Author

E-mail: xxjiao@ucas.ac.cn; \phantom{,,} cuihongbin16@mails.ucas.ac.cn; \phantom{,,}xinjialin17@mails.ucas.ac.cn

\end{center}

\bigskip

\no
{\bf Abstract.}
This paper is the continuation of the previous one \cite{Cui2021}, where we re-proved the area-minimization of cones over Grassmannians of $n$-planes $G(n,m;\mathbb{F})(\mathbb{F}=\mathbb{R},\mathbb{C},\mathbb{H})$, Cayley plane $\mathbb{O}P^2$ from the point view of Hermitian orthogonal projectors, and gave area-minimizing cones associated to oriented real Grassmannians $\widetilde{G}(n,m;\mathbb{R})$.

In this paper, we make a further step on showing that the cones, of dimension no less than $\mathbf{8}$, over minimal products of $G(n,m;\mathbb{F})$ are area-minimizing. Moreover, those cones are very similar to the classical cones over products of spheres, and for the critical situation---the cones of dimension $\mathbf{7}$ \cite{lawlor1991sufficient}, we gain more area-minimizing cones by carefully computing the Jacobian $inf_{v}det(I-tH^{v}_{ij})$. Certain minimizing cones among them had been found from the perspective of $R$-spaces\cite{Ohno2021area}, or isoparametric theory\cite{tang2020minimizing}, and others are completely new.

We also prove that the cones over minimal product of $\widetilde{G}(n,m;\mathbb{R})$ are area-minimizing.\\


\no

\bigskip

\tableofcontents
\section{Introduction}
Area-minimizing cones are a class of area-minimizing surfaces whose its truncated part inside the unit ball owns the least area among all integral currents with the same boundary in the sphere.

With an effort for searching necessary and sufficient conditions for a cone to be area-minimizing, Gary R. Lawlor(\cite{lawlor1991sufficient}) developed a general method for proving that a cone is area-minimizing, which was called Curvature Criterion by himself.

Lawlor explained his Curvature Criterion by two equivalent concrete objects, the vanishing calibration and the area-nonincreasing retraction, both defined on certain angular neighborhood of the cone rather than in the entire Euclidean space. They are linked by the fact that the tangent space of retraction surface is just the orthogonal complement of the dual of the vanishing calibration. They both derive an ordinary differential equation($ODE$). The simplified area-minimizing tests include if the $ODE$ has solutions, what is the maximal existence interval of a solution, and then compare it with an important potential of the cone-\textbf{normal radius}. Lawlor also studied those cones which his Curvature Criterion is both necessary and sufficient like the minimal, isoparametric hypercones and the cones over principal orbits of polar group actions. Based on his method, the classification of minimal(area-minimizing, stable, unstable) cones over products of spheres is completed, he also proved some cones over unorientable real projective spaces, compact matrix groups are area-minimizing and gave new proofs of a large class of homogeneous hypercones being area-minimizing.

Other related researches, or researches for the area-minimizing cones associated to Lawlor's Curvature Criterion are: a new twistor calibrated theory and applications for the Veronese cone given by Tim Murdoch\cite{murdoch1991twisted}, also see an new point view of Veronese cone from vanishing calibration\cite{lawlor1995note}; Extending the definition of coflat calibration and the illustration of special Largraigian calibrated cones over compact matrix group which are also shown in \cite{lawlor1991sufficient} are given by Benny N. Cheng\cite{cheng1988area}; Proofs for area-minimization of $\textbf{Lawson-}$
$\textbf{Osserman}$ cones are given by Xiaowei Xu, Ling Yang and Yongsheng Zhang\cite{xu2018new}; Researches on area-minimizing cones associated with isoparametric foliations has been pioneered by Zizhou Tang and Yongsheng Zhang\cite{tang2020minimizing}, etc.

The Grassmannian manifolds are important symmetric spaces which can be endowed with some special geometric structures. Generally, the research objects include the Grassmannian of $n$-plane in $\mathbb{F}^{m}$: $G(n,m;\mathbb{F})$
($\mathbb{F}=\mathbb{R},\mathbb{C},\mathbb{H}$), the Grassmannian of oriented $n$-plane in $\mathbb{R}^{m}$: $\widetilde{G}(n,m;\mathbb{R})$, and there exists an only analogue for $\mathbb{F}=\mathbb{O}$, the Cayley plane: $\mathbb{O}P^2$.

There exists standard embedding maps for those Grassmannian manifolds, except those $\widetilde{G}(n,m;\mathbb{R})(n,m-n\neq 2)$, into Euclidean space which can be seen as isolated orbits of some polar group actions, or standard embedding of symmetric $R$-spaces(\cite{berndt2016submanifolds}). Those maps are minimal, the images lies in the spheres, then an natural question is to ask: Does those minimal cones over the images area-minimizing, stable or unstable?

By considering $G(n,m;\mathbb{R})$ and $G(n,m;\mathbb{C})$ as isolated singular orbits of adjoint actions of special orthogonal groups and special unitary groups(\cite{kerckhove1994isolated}), Michael Kerckhove proved almost all their cones are area-minimizing. From the perspective of symmetric $R$-spaces and their canonical embeddings for $\widetilde{G}(2,2l+1;\mathbb{R})(l\geq 3)$, by constructing area-nonincreasing retractions directly, Daigo Hirohashi, Takahiro Kanno and Hiroyuki Tasaki(\cite{hirohashi2000area}) proved the area-minimization of these cones.

Later on, also from the point view of symmetric $R$-spaces, Takahiro Kanno(\cite{kanno2002area}) proved that all the cones over the Grassmannian of subspaces $G(n,m;\mathbb{F})$(where $\mathbb{F}=\mathbb{R},\mathbb{C},\mathbb{H}$, and except $\mathbb{R}P^2$) and Grassmannian of oriented $2$-planes $\widetilde{G}(2,m;\mathbb{R})(m\geq 5)$ are area-minimizing.

The cone over $\mathbb{O}P^2$ was also shown area-minimizing from the above point view by Shinji Ohno and Takashi Sakai(\cite{Ohno2021area}) as part of their results. Moreover, based on the simple fact: the Riemannian product of minimal immersion(embedding) maps into unit spheres is also a minimal immersion(embedding) into unit sphere after some digital adjustment, they proved that those cones over minimal products of R-spaces which each R-space associate to an minimizing cone in their table(obviously ruled out $\mathbb{R}P^2$) are also area-minimizing.

We also note here, standard embeddings of $\mathbb{F}P^2(\mathbb{F}=\mathbb{R},\mathbb{C},\mathbb{H},\mathbb{O})$ into unit spheres of Euclidean spaces are often called Veronese embeddings, they are focal submanifolds of isoparametric foliations with the number of principal curvatures $g=3$ and of multiplicity $m_{1}=m_{2}=m_{3}=1,2,4,8$. Some embedded oriented real Grassmannians also belong to focal submanifolds associated to isoparametric foliations. In \cite{tang2020minimizing},  Zizhou Tang and Yongsheng Zhang proved that the cones over minimal products of focal submanifold of isoparametric foliations for $g=3,4,6$ and $(m_{1},m_{2})\neq (1,1)$ are area-minimizing.

In an early paper\cite{Cui2021}, we re-proved the area-minimization of cones over Grassmannians of $n$-planes $G(n,m;\mathbb{F})(\mathbb{F}=\mathbb{R},\mathbb{C},\mathbb{H})$, Cayley plane $\mathbb{O}P^2$ from the point view of Hermitian orthogonal projectors, and gave area-minimizing cones associated to the Pl\"{u}cker embedding of oriented real Grassmannians $\widetilde{G}(n,m;\mathbb{R})$. Here, we continue to give a class of area-minimizing cones over products of those Grassmannian manifolds which inspired by the minimal product considered in \cite{Ohno2021area}, \cite{tang2020minimizing}, \cite{choe2018some} independently.

\vspace{0.5cm}
$\mathbf{Statement \ of \ Results}$
Under standard embeddings, every Grassmannian under a base field $\mathbb{F}$( $\mathbb{F}=\mathbb{R},\mathbb{C},\mathbb{H}$) is a minimal submanifold in associated dimensional spheres, and for their minimal product, we prove that

\begin{theorem}\label{high dim}
The cone $C$ over $M=\prod_{i=1}^{m}\sqrt{k_{i}}G(l_{i},k_{i};\mathbb{F})$ is area-minimizing, if $dim C>7$.
\end{theorem}

The base fields of Grassmannians can also be different,

\begin{theorem}
The cone $C$ over $M=\prod_{i=1}^{m}\sqrt{d_{i}k_{i}}G(l_{i},k_{i};\mathbb{F}_{i})$ is area-minimizing if $dim C>7$, where $d_{i}=dim_{\mathbb{R}}\mathbb{F}_{i}$.
\end{theorem}

Moreover, the cones of dimension $7$ own the same critical properties which exhibited in \cite{lawlor1991sufficient} for the cases of product spheres, similar, we need to find the concrete expressions for the Jacobian $inf_{v}det(I-tH^{v}_{ij})$ in \eqref{VN}, based on the work of \cite{lawlor1991sufficient}, we successfully establish

\begin{theorem}\label{general seven}
The cones over standard embeddings of types (I)$\mathbb{R}P^2 \times \mathbb{R}P^2  \times \mathbb{R}P^2$, $\mathbb{R}P^2 \times G(2,4;\mathbb{R})$, $\mathbb{R}P^2 \times \mathbb{R}P^4$, $\mathbb{R}P^3 \times \mathbb{R}P^3$, $\mathbb{R}P^2 \times \mathbb{C}P^2$, and moreover (II)$S^2 \times S^2 \times \mathbb{R}P^2$, $S^2 \times \mathbb{R}P^2 \times \mathbb{R}P^2$, $S^2 \times \mathbb{C}P^2$, $S^2 \times \mathbb{R}P^4$, $S^2 \times G(2,4;\mathbb{R})$, $S^4 \times \mathbb{R}P^2$, $S^3 \times \mathbb{R}P^3$ are area-minimizing.
\end{theorem}

\begin{remark}
Links in Types (I) are product manifolds where each factors are embedded Grassmannian submanifolds, Types (II) are product manifolds which owns reduced factors, and they are all at least two-dimensional spheres.
\end{remark}

Our investigation for the minimal products of oriented real Grassmannians leads to
\begin{theorem}
Cones over minimal product of $Pl\ddot{u}cker \ embedded$, oriented real Grassmannians are area-minimizing.
\end{theorem}

\begin{remark}
In \cite{Ohno2021area}, they proved that those cones over minimal products of R-spaces which each R-space associate to an minimizing cone in their table(ruled out $\mathbb{R}P^2$) are also area-minimizing. In contrast, the factor of Veronese embedded $\mathbb{R}P^2$ is not included in their results, while some embedded oriented real Grassmannians were shown in their results as $R$-spaces. And for Grassmannian of oriented $2$-planes(can be identified with complex hyperquadric), our results are coincided with theirs. Similar results are also established in \cite{tang2020minimizing} independently from the perspective of isoparametric theory.
\end{remark}

This article is organized as follows:

In chapter 2, we give illustration of minimal product and the concrete formula for its norm of sharp operator. Then, we review some known results for the standard embedding of Grassmannians into spheres and part of the work of Gary R. Lawlor in \cite{lawlor1991sufficient}.

In chapter 3, we study the minimal product of Grassmannian manifolds, then we give the proof for area-minimizing of cones which own the dimension greater than 7.

In chapter 4, we prove Theorem \ref{general seven} through detailed discussions for the minimum values of Jacobian $det(I-tH_{ij}^{v})$.

In chapter 5, we prove the area-minimization for cones over minimal product of oriented real Grassmannians.
\section{Preliminaries}
\subsection{Minimal products}

Given a family of minimal immersion $f_{i}:M_{i}^{k_{i}}\hookrightarrow S^{a_{i}}(1)$($1\leq i \leq m$), the standard Euclidean inner product is denoted by $\langle,\rangle$, then $\langle f_{i},f_{i} \rangle=1$, denote the second fundamental form of $f_{i}$ with respect to the normal vector $v_{i}$ by $H_{i}^{v_{i}}=-\langle df_{i},dv_{i}\rangle$.

Consider the product immersion,
\begin{equation}
\begin{aligned}
f:M=M_{1}^{k_{1}}\times \cdots \times M_{m}^{k_{m}} &\rightarrow S^{a_{1}+\cdots+ a_{m}+m-1}(1)\\
(x_{1},\ldots,x_{m})&\mapsto (\lambda_{1}f_{1}(x_{1}),\ldots,\lambda_{m}f_{m}(x_{m})),
\end{aligned}
\end{equation}
where $\sum_{i=1}^{m}\lambda_{i}^{2}=1$, $dim \ M=k_{1}+\cdots+ k_{m}$.

Now, $df=(\lambda_{1}df_{1},\ldots,\lambda_{m}df_{m})$, set $v=(v_{1},\ldots,v_{m})$, then $v$ is a normal vector of $M$ if and only if $\langle df,v\rangle =0$ and $\langle  f,v\rangle =0$, i.e. in  every Euclidean subspaces $\mathbb{R}^{a_{i}+1}$, we set $v_{i}=\xi_{i}-b_{i}x_{i}$, here $x_{i}$ denote the position vector $f_{i}(x_{i})$, and it should satisfy
\begin{equation}
\sum_{i=1}^{m}\lambda_{i}b_{i}=0.
\end{equation}

Let $H^{v}$ denote the second fundamental forms of $f$ with respect to the normal vector $v$, then
\begin{equation}
\begin{aligned}
H^{v}&=-\langle df,dv\rangle \\
&=-\sum_{i=1}^{m}(\lambda_{i}\langle df_{i},d\xi_{i} \rangle- \lambda_{i}b_{i}\langle dx_{i},dx_{i}\rangle) \\
&=\sum_{i=1}^{m}(\lambda_{i}H_{i}^{\xi_{i}}+\lambda_{i}b_{i}g_{i}),
\end{aligned}
\end{equation}
where $g_{i}=\langle dx_{i},dx_{i}\rangle$ is the induced metric of immersion $f_{i}$.

\begin{prop}
The product $f$ of minimal immersions $f_{i}$($1\leq i \leq m$) is also a minimal immersion if and only if $\lambda_{i}=\sqrt{\frac{k_{i}}{\sum_{i=1}^{m}k_{i}}}$.
\end{prop}

\begin{proof}
Let $\{e_{i1},\ldots,e_{ik_{i}}\}$ be an orthonormal basis of $(M_{i},g_{i})$, then $\{\ldots,\frac{e_{i1}}{\lambda_{i}},\ldots,\frac{e_{ik_{i}}}{\lambda_{i}},\ldots\}$ is an orthonormal basis of $M$, hence
\begin{equation}
H^{v}\left(\frac{e_{il}}{\lambda_{i}},\frac{e_{is}}{\lambda_{i}}\right)=\sum_{i=1}^{m}\frac{b_{i}
+H_{i}^{\xi_{i}}(e_{il},e_{is})}{\lambda_{i}},
\end{equation}
 where $l,s\in \{1,\ldots,k_{i}\}$, and under the above orthonormal basis, the matrix of $H^{v}$ is given by
\begin{equation}\label{base equation}
H^{v}=diag \{\frac{b_{1}}{\lambda_{1}}I_{k_{1}}+\frac{1}{\lambda_{1}}H_{1}^{\xi_{1}},\ldots,
\frac{b_{m}}{\lambda_{m}}I_{k_{m}}+\frac{1}{\lambda_{m}}H_{m}^{\xi_{m}}\},
\end{equation}

where $H_{i}^{\xi_{i}}$ is the matrix of second fundamental forms of $f_{i}$ on the direction $\xi_{i}$.

Hence, $f$ is a minimal immersion if and only if
\begin{equation} \label{product minimal}
tr \ H^{v}=\sum_{i=1}^{m}\frac{b_{i}k_{i}}{\lambda_{i}}=0
\end{equation}
for any chosen normal vector $v, v_{i}=\xi_{i}-b_{i}x_{i}$, since every $H_{i}^{\xi_{i}}$ has trace zero.

The left proof is similar to \cite{lawlor1991sufficient}, if we choose $\lambda_{i}=\sqrt{\frac{k_{i}}{\sum_{i=1}^{m}k_{i}}}=\sqrt{\frac{dim M_{i}}{dim M}}$, then \eqref{product minimal} holds.
\end{proof}

We continue to compute the norm of $H^{v}$,

\begin{equation}
\begin{aligned}
|H^{v}|^2&=\sum_{i=1}^{m}\frac{k_{i}b_{i}^2+|H_{i}^{\xi_{i}}|^2+2b_{i}tr\ H_{i}^{\xi_{i}}}{\lambda_{i}^2} \\
&=dim \ M\left(\sum_{i}b_{i}^2+\sum_{i}\frac{|\xi_{i}|^2}{k_{i}}|H_{i}^{\frac{\xi_{i}}{|\xi_{i}|}}|^2\right) \\
&=dim \ M\left(1+\sum_{i}|\xi_{i}|^2\left(\frac{1}{k_{i}}|H_{i}^{\frac{\xi_{i}}{|\xi_{i}|}}|^2-1\right)\right) \\
&=dim \ M\cdot max\{1,\frac{\alpha_{1}^2}{k_{1}},\ldots,\frac{\alpha_{m}^2}{k_{m}} \}.
\end{aligned}
\end{equation}
where $\alpha_{i}^2:=sup_{\xi}|H_{i}^{\xi}|^2$, $\xi$ is a unit normal vector of $M_{i}$, $k_{i}=dim \ M_{i}$.

\begin{theorem}\label{forms formula}
The upper bound of second fundamental forms of $f$ is given by $\alpha^2:=sup_{v}|H^{v}|^2=dim \ M\cdot max\{1,\frac{\alpha_{1}^2}{k_{1}},\ldots,\frac{\alpha_{m}^2}{k_{m}} \}$.
\end{theorem}

\subsection{Standard Embedding of Grassmannians into spheres}

This subsection is based on \cite{chen2014total}, a detailed discussion for the spherical, standard embedding of Grassmannians into the Euclidean spaces consists of the Hermitian matrices is given in \cite{Cui2021}.

Let $\mathbb{F}$ denote the field of real numbers $\mathbb{R}$, the field of complex numbers $\mathbb{C}$, the normed quaternion associative algebra $\mathbb{H}$, $d=dim_{R}\mathbb{F}=\{1,2,4\}$.

We use the following notations: $M(m;\mathbb{F})$ denote the space of all $m\times m$ matrices over $\mathbb{F}$, $H(m;\mathbb{F})\subset M(m;\mathbb{F})$ denote the space of all Hermitian matrices over $\mathbb{F}$,
\begin{equation}
H(m;\mathbb{F})=\{A\in M(m;\mathbb{F})|A^{*}=A\}.
\end{equation}

$H(m;\mathbb{F})$ can be identified with real Euclidean space $E^{N}$ with the inner product: \begin{equation}
g(A,B)=\frac{1}{2} Re\ tr_{\mathbb{F}}(AB),
\end{equation}
where $A,B\in H(m;\mathbb{F})$,  $N=m+dm(m-1)/2$.

Let $U(m;\mathbb{F})=\{A\in M(m;F)|AA^{*}=A^{*}A=I\}$ denote the $\mathbb{F}$-unitary group, $U(m;\mathbb{F})$ has an natural adjoint action $\rho$ on $H(m;\mathbb{F})$, given by $\rho(Q,P)=QPQ^{*}$, where $Q\in U(m;\mathbb{F}), P\in H(m;\mathbb{F})$, this action is isometry and transitive.

Let $G(n,m;\mathbb{F})$ denote the Grassmannian of $n$-plane in the right vector space $\mathbb{F}^{m}$, for every $L\in G(n,m;\mathbb{F})$, there exists an Hermitian orthogonal projector $P_{L}$ associated with it, hence give an embedding $\varphi$ of $G(n,m;\mathbb{F})$ into the hypersphere contained in $H(m;\mathbb{F})$(the dimension of the sphere can reduce to $N-2$):
\begin{equation}
\begin{aligned}
\varphi:G(n,m;\mathbb{F})&\rightarrow H(m;\mathbb{F})\\
 L &\mapsto P_{L},
\end{aligned}
\end{equation}
where $P_{L}=P_{L}^{*}$,$P_{L}^{2}=P_{L}$, $tr\ P_{L}=n$ and $L=\{z\in\mathbb{F}^{m}|P_{L}z=z\}$.

This embedding is minimal, and it's cone was shown area-minimizing in \cite{Cui2021}, other perspectives for these cones being area-minimizing can be seen in \cite{kerckhove1994isolated},\cite{hirohashi2000area},
 \cite{kanno2002area},\cite{Ohno2021area}.

\subsection{Gary R. Lawlor's table of Vanishing angles}
Now, we recall Gary R.Lawlor's work in \cite{lawlor1991sufficient}, for the following $ODE$ (see definition 1.1.6 in \cite{lawlor1991sufficient}):
\begin{equation}\label{VN}
\begin{cases}
\frac{dr}{d\theta}=r \sqrt{r^{2k}(cos\theta)^{2k-2}\mathrm{inf}_{v}\left(det(I-tan(\theta) h_{ij}^{v})\right)^{2}-1}\\
r(0)=1,
\end{cases}
\end{equation}
where $h_{ij}^{v}$ is the matrix representation of the second fundamental form of an minimal submanifold $M$ in sphere, $v$ is an unit normal, $k$ is the dimension of cone $C=C(M)$, and $r=r(\theta)$ describe a projection curve, the $ODE$ is build at a fixed point $p\in M$.

Denote the real vanishing angle by $\theta_{0}$(see Definition 1.1.7 in \cite{lawlor1991sufficient}), Lawlor use the following estimates.

Let $\theta_{1}(k,\alpha)$ be the estimated vanishing angle function which replacing $\mathrm{inf}_{v}det(I-tan(\theta) h_{ij}^{v})$ by an smaller positive-valued function
\begin{equation}
F(\alpha,tan(\theta),k-1)=\left(1-\alpha tan(\theta)\sqrt{\frac{k-2}{k-1}}\right)\left(1+\frac{\alpha tan(\theta)}{\sqrt{(k-1)(k-2)}}\right)^{k-2}
\end{equation}
in \eqref{VN}, where the condition $\alpha^2 tan^2(\theta_{1})\leq\frac{k-1}{k-2}$
should be satisfied, and $F$ is an decreasing function of $\alpha$ when $tan(\theta),k$ are fixed, it is also decreasing with respect to $k$ when $\alpha,tan(\theta)$ are fixed.

Let $\theta_{2}(k,\alpha)$ be the estimated vanishing angle function which replacing $\mathrm{inf}_{v}det(I-tan(\theta) h_{ij}^{v})$ by
\begin{equation}
lim_{k\rightarrow \infty} F(\alpha,tan(\theta),k-1)=\left(1-\alpha tan(\theta)\right)e^{\alpha tan(\theta)}
\end{equation}
in \eqref{VN}, where the condition $\alpha^2 tan^2(\theta_{2})\leq1$
should be satisfied, and $(1-\alpha tan(\theta))e^{\alpha tan(\theta)}$ is also an decreasing function of $\alpha$ when $tan(\theta)$ are fixed.

The three angles have the following relation:
\begin{equation}
\theta_{0}\leq \theta_{1}(k,\alpha)\leq \theta_{2}(k,\alpha),
\end{equation}
and Lawlor use the angle function $\theta_{1}$ for $dim\ C=\{3,\ldots,11\}$, the angle function $\theta_{2}$ for $dim\ C=12$ to gain 'The Table'(see section 1.4 in \cite{lawlor1991sufficient}).

A simplified version of Curvature Criterion states that if two times of estimated vanishing angles are still less than the associated normal radius pointwise, then these cones are area-minimizing.

\section{Cones of dimension bigger than $7$}
Now, for every $i\in \{1,\ldots,m\}$, consider the minimal embedding(\cite{Cui2021})
\begin{equation}
\begin{aligned}
G(l_{i},k_{i};\mathbb{F})&\rightarrow S^{n_{i}}(r_{i})\\
 A &\mapsto A-\frac{l_{i}}{k_{i}}I,
\end{aligned}
\end{equation}
where $r_{i}=\sqrt{\frac{l_{i}(k_{i}-l_{i})}{2k_{i}}}$, $dim\ G(l_{i},k_{i};\mathbb{F})=dl_{i}(k_{i}-l_{i})$, $n_{i}=k_{i}-2+\frac{k_{i}(k_{i}-1)}{2}d$, and we assume $l_{i}\leq k_{i}-l_{i}$ for every $i$, it contain the following three reduced cases: $\mathbb{R}P^{1}\equiv S^{1}(\frac{1}{2})$, $\mathbb{C}P^{1}\equiv S^{2}(\frac{1}{2})$ and $\mathbb{H}P^{1}\equiv S^{4}(\frac{1}{2})$(\cite{Cui2021}).

First, we consider that a family of embedded Grassmannians are of the same base field. Let $f_{i}$ be the normalized minimal embedding $f_{i}:\frac{1}{r_{i}}G(l_{i},k_{i};\mathbb{F})\rightarrow S^{n_{i}}(1)$, then $f:M\equiv\prod_{i=1}^{m}\frac{\lambda_{i}}{r_{i}}G(l_{i},k_{i};\mathbb{F})\rightarrow S^{\sum_{i=1}^{m}n_{i}+m-1}(1)$ is minimal if and only if $a_{i}=\frac{\lambda_{i}}{r_{i}}=\sqrt{\frac{2dk_{i}}{dim \ M}}$, where $dim \ M=\sum_{i=1}^{m}dl_{i}(k_{i}-l_{i})$.

Hence,
\begin{prop}
The cone over $\prod_{i=1}^{m}\sqrt{k_{i}}G(l_{i},k_{i};\mathbb{F})$ is minimal.
\end{prop}

We will use Gary R. Lawlor's Curvature Criterion to verify that whether these cones are area-minimizing.

\begin{prop}
The upper bound of the second fundamental forms of the cone over $\prod_{i=1}^{m}\sqrt{k_{i}}G(l_{i},k_{i};\mathbb{F})$ at the points belong to the unit sphere is given by $sup_{v}|H^{v}|^2=dim \ M$.
\end{prop}
\begin{proof}
The product of embedded Grassmannians can be seen as an orbit of product of associated unitary group(\cite{kerckhove1994isolated}), then we can compute the second fundamental forms at an fixed point.

Since $sup_{v}|H^{v}|^2=dim \ M \cdot max\{1,\frac{\alpha_{1}^2}{dl_{1}(k_{1}-l_{1})},\ldots,\frac{\alpha_{m}^2}{dl_{m}(k_{m}-l_{m})}\}$,  and $\alpha_{i}^2=\frac{dl_{i}(k_{i}-l_{i})^2}{k_{i}}$(\cite{kerckhove1994isolated},\cite{Cui2021}), then $\frac{\alpha_{i}^2}{dl_{i}(k_{i}-l_{i})}=\frac{k_{i}-l_{i}}{k_{i}}<1$.
\end{proof}

The computation of normal radius follows \cite{kerckhove1994isolated},\cite{Cui2021} directly, in the $i$-factor of $P=\left(\ldots,a_{i}(1-\frac{l_{i}}{k_{i}}),\ldots,a_{i}(1-\frac{l_{i}}{k_{i}}),a_{i}(-\frac{l_{i}}{k_{i}}),
\ldots,a_{i}(-\frac{l_{i}}{k_{i}}),\ldots\right)$, we exchange one pair of $a_{i}(1-\frac{l_{i}}{k_{i}})$ and $a_{i}(-\frac{l_{i}}{k_{i}})$ in the same one factor to gain another point $\tilde{P}$, and
\begin{equation}
\langle P,\tilde{P}\rangle=1-\frac{dk_{i}}{dim \ M}.
\end{equation}

Hence,
\begin{prop}\label{normal radius}
The normal radius of the cone over $\prod_{i=1}^{m}\sqrt{k_{i}}G(l_{i},k_{i};\mathbb{F})$ is $arccos(1-\frac{dk_{1}}{dim \ M})$, where $k_{1}=min \{k_{1},\ldots,k_{m}\}$.
\end{prop}

The left work is to compare the normal radius and the estimated vanishing angles, we have
\begin{theorem}\label{high dim}
The cone $C$ over $M=\prod_{i=1}^{m}\sqrt{k_{i}}G(l_{i},k_{i};\mathbb{F})$ is area-minimizing, if $dim C>7$.
\end{theorem}
\begin{proof}
We first consider the cases: $7\leq dim M\leq 11$.

~(1)$\mathbb{F}=\mathbb{H}$, M can only be $\mathbb{H}P^1 \times \mathbb{H}P^1$. $(dim C,\alpha^2)=(9,8)$, then the estimated vanishing angle is $12.99^{\circ}$ by Lawlor's table, and the normal radius is $arccos(1-\frac{dk_{1}}{dim M})=\frac{\pi}{2}$, it is area-minimizing;

~(2)$\mathbb{F}=\mathbb{C}$, all of the normal radius between different cases have the minimum value: $arccos(1-\frac{4}{10})=arccos(\frac{3}{5})=53^{\circ}$, and the maximal value of vanishing angle is no more than $12.99^{\circ}$ by Lawlor's table, so these cones are also area-minimizing;

~(3)$\mathbb{F}=\mathbb{R}$, all of the normal radius between different cases have the minimum value: $arccos(1-\frac{2}{11})=arccos(\frac{9}{11})=35^{\circ}$, and the maximal value of vanishing angle is no more than $15.84^{\circ}$ by Lawlor's table, so these cones are also area-minimizing.

For $dim M\geq 12$, let $k=dim C=dim M+1\geq 13$, we use the following formula given by Gary R. Lawlor:
\begin{equation}
tan(\theta_{2}(k,\alpha))< \frac{12}{k}tan\left(\theta_{2}\left(12,\frac{12}{k}\alpha\right)\right).
\end{equation}

Now, $\alpha=\sqrt{k-1}$, then $\frac{12}{k}\alpha\leq \sqrt{13}$, hence
\begin{equation}
tan\left(\theta_{2}(k,\alpha)\right)< \frac{12}{k}tan\left(\theta_{2}(12,\sqrt{13})\right)<\frac{2}{k}.
\end{equation}

The normal radius of cone over $M=\prod_{i=1}^{m}\sqrt{k_{i}}G(l_{i},k_{i};\mathbb{F})$ is $arccos(1-\frac{dk_{1}}{k-1})$(we arrange the factors such that $k_{1}=min \{k_{1},\ldots,k_{m}\}$), we compare $2arctan\frac{2}{k}$ and $arccos(1-\frac{dk_{1}}{k-1})$ as follows:

Note $2arctan\frac{2}{k}<\frac{4}{k}$, $1-cos\frac{4}{k}<\frac{8}{k^2}$ and $dk_{1}\geq 2$, then the proof is followed by the above relations.
\end{proof}

There are also analogies for the product embedding of Grassmannian manifolds over different base fields, consider the normalized minimal embeddings $f_{i}:\frac{1}{r_{i}}G(l_{i},k_{i};\mathbb{F}_{i})\rightarrow S^{n_{i}}(1)$, set $d_{i}=dim_{\mathbb{R}}\mathbb{F}_{i}$, then $r_{i}=\sqrt{\frac{l_{i}(k_{i}-l_{i})}{2k_{i}}}$, $dim\ G(l_{i},k_{i};\mathbb{F}_{i})=d_{i}l_{i}(k_{i}-l_{i})$, $n_{i}=k_{i}-2+\frac{k_{i}(k_{i}-1)}{2}d_{i}$, and we assume $l_{i}\leq k_{i}-l_{i}$ for every $i$.

Let $M$ be $\prod_{i=1}^{m}\frac{\lambda_{i}}{r_{i}}G(l_{i},k_{i};\mathbb{F}_{i})$, then $f:M\rightarrow S^{\sum_{i=1}^{m}n_{i}+m-1}(1)$ is minimal if and only if $a_{i}=\frac{\lambda_{i}}{r_{i}}=\sqrt{\frac{2d_{i}k_{i}}{dim \ M}}$, where $dim \ M=\sum_{i=1}^{m}d_{i}l_{i}(k_{i}-l_{i})$.

Hence,
\begin{prop}
The cone over $\prod_{i=1}^{m}\sqrt{d_{i}k_{i}}G(l_{i},k_{i};\mathbb{F}_{i})$ is minimal.
\end{prop}

The upper bound of second fundamental forms of this cone is also $dim M$, its normal radius is $arccos (1-\frac{min\{d_{i}k_{i}\}}{dim \ M})$.  After an similar discussion to theorem \ref{high dim}, we have
\begin{theorem}
The cone $C$ over $M=\prod_{i=1}^{m}\sqrt{d_{i}k_{i}}G(l_{i},k_{i};\mathbb{F}_{i})$ is area-minimizing if $dim C>7$, where $d_{i}=dim_{\mathbb{R}}\mathbb{F}_{i}$.
\end{theorem}

\section{Cones of dimension $7$}
In this section, we consider $M$ as the product of embedded Grassmannian manifolds of $dim M=6$, just like the cases of products of spheres(see chapter $4$ and $5$ in \cite{lawlor1991sufficient}), the cone over $M$ may still have vanishing angles, we should find the concrete expressions of $inf_{v}det(I-tH^{v}_{ij})$ in \eqref{VN}.

Let $A$ be an arbitrary $m \times m$ symmetric matrix whose trace is zero, $\alpha:=||A||=\sqrt{\sum a_{ij}^2}$, $A$ just has $m$ real eigenvalues $\{a_{1},\ldots,a_{m}\}$, then $det(I-tA)$ attains its minimum for any symmetric matrix $A$ when there are only two different values represented among all of the $a_{i}$((see Appendix of \cite{lawlor1991sufficient}), and the conditions $\sum a_{i}=0$ and $\sum a_{i}^2=\alpha^2$ completely determine the solution once we know how many should be positive. Denote the multiplicity of the positive eigenvalue by $r$, the solution is $\alpha \sqrt{\frac{m-r}{mr}}$ of multiplicities $r$ and $-\alpha \sqrt{\frac{r}{m(m-r)}}$ of multiplicities $m-r$.

Let $L(\alpha,t,m,r)$ be the result function:
\begin{equation}
L(\alpha,t,m,r):=\left(1-t\alpha \sqrt{\frac{m-r}{mr}}\right)^{r}\left(1+t\alpha \sqrt{\frac{r}{m(m-r)}}\right)^{m-r},
\end{equation}
where $r\in [1,m-1]$ is an integer, $L$ is decreasing in $r$, the minimum of $L$ with respect to $r$ is denoted by $F(\alpha,t,m):=L(\alpha,t,m,1)$.

The smaller $r$ is, the bigger the angle radius of normal wedge will be. Lawlor use the function $F(\alpha,t,m)$ to build "The table" in \cite{lawlor1991sufficient}. There's a critical situation that when $\alpha=\sqrt{6}, m=6$, from "The Table", the solution of \eqref{VN} associated to $inf_{v}det(I-tH^{v}_{ij})=F(\sqrt{6},t,6)$ exists in some finite interval $[0,\theta]$, the vanishing angle does not exist. Though, for concrete cones under critical situation, the multiplicity of the positive eigenvalue may fail to attain the least number $1$, the cone over $M$ may still have vanishing angles.

The multiplicity of the positive eigenvalue played key roles in the complete classifcation of cones over products of spheres, for $\alpha=\sqrt{6}, m=6$, the function of multiplicity $1$ is
\begin{equation}
F(t):=F(\sqrt{6},t,6)=L(\sqrt{6},t,6,2)=(1-t\sqrt{5})\left(1+\frac{t}{\sqrt{5}}\right)^{5},
\end{equation}
it associate to the products of spheres which one of the spheres is a circle, their cones are stable, by the studying of cones for which the Curvature Criterion is Necessary and Sufficient, Lawlor conclude that these cones are stable, not area-minimizing.

The function of multiplicity $2$ is
\begin{equation}
E(t):=L(\sqrt{6},t,6,2)=(1-t\sqrt{2})^{2}\left(1+\frac{t}{\sqrt{2}}\right)^{4},
\end{equation}
associated to the cones over $S^2\times S^4$ and $S^2\times S^2 \times S^2$, their cones have vanishing angles, hence be area-minimizing.

And the function of multiplicity $3$ is
\begin{equation}
G(t):=L(\sqrt{6},t,6,3)=(1-t)^{3}(1+t)^{3},
\end{equation}
associated to the cone over $S^3\times S^3$, the cone is area-minimizing too.

Now, we consider the cones $C=C(M)$ over the product of embedded Grassmannian manifolds of dimension $7$. Under the standard embedding, $\mathbb{R}P^{1}\equiv S^{1}(\frac{1}{2})$, $\mathbb{C}P^{1}\equiv S^{2}(\frac{1}{2})$ and $\mathbb{H}P^{1}\equiv S^{4}(\frac{1}{2})$(\cite{Cui2021}), we call their associated factors in the product---the $reduced \ factors$ since there doesn't exist the terms of second fundamental forms in \eqref{base equation}, or additionally, the factors could be any dimensional spheres less than $6$,  i.e. one of the factors is $S^1,S^2,S^3,S^4,S^5$.

So, there exists some cones over the products of embedded Grassmannian manifolds which part of its factors are $reduced \ factors$, and there are five cones which all its factors are not $reduced \ factors$, we divide them into two types, and we will show that cones over types:

~(I)$\mathbb{R}P^2 \times \mathbb{R}P^2  \times \mathbb{R}P^2$, $\mathbb{R}P^2 \times G(2,4;\mathbb{R})$, $\mathbb{R}P^2 \times \mathbb{R}P^4$, $\mathbb{R}P^3 \times \mathbb{R}P^3$, $\mathbb{R}P^2 \times \mathbb{C}P^2$, and

~(II)$S^2 \times S^2 \times \mathbb{R}P^2$, $S^2 \times \mathbb{R}P^2 \times \mathbb{R}P^2$, $S^2 \times \mathbb{C}P^2$, $S^2 \times \mathbb{R}P^4$, $S^2 \times G(2,4;\mathbb{R})$, $S^4 \times \mathbb{R}P^2$, $S^3 \times \mathbb{R}P^3$ are area-minimizing.

We will prove that these cones are all area-minimizing by computing their minimum polynomials which none of them are multiplicity $1$ polynomial $F(t)$,

(1)~~$\mathbb{R}P^2 \times \mathbb{R}P^2  \times \mathbb{R}P^2$($S^2 \times S^2 \times \mathbb{R}P^2$, $S^2 \times \mathbb{R}P^2 \times \mathbb{R}P^2$):

First, we give illustrations for the area-minimization of the cone over the first one, then another one is an natural generalization.

The first submanifold is the products of three Veronese maps, for this case, we will prove that on $(0,\frac{2\sqrt{2}}{7})$, for any chosen unit normal vector $v$,
\begin{equation}
inf_{v}det(I-tH^{v}_{ij})=(1-t\sqrt{2})^2\left(1+\frac{t}{\sqrt{2}}\right)^4=E(t).
\end{equation}

Note $\frac{2\sqrt{2}}{7}=0.404$, and the vanishing angle associated to $E(t)$ is $tan19.9^{\circ}=0.362$, by considering its normal radius, we conclude that this cone is area-minimizing.

The illustrations are as follows:

First, $\mathbb{R}P^2$ is embedded in $H(3,\mathbb{R})$(or its hyperplane of trace zero), the normal vector $\xi_{1}$ is given by $diag\{0,a,-a\}$(to see this, first assume $\xi_{1}$ is an any given normal vector, if $g_{\ast}\xi_{1}:=g\xi_{1}g^{T}$ is diagonal for an isotropy isometry action $g$, we could change the standard orthonormal tangent basis $E_{\alpha}:=E_{\alpha 1}+E_{1 \alpha}(2\leq \alpha \leq 3)$ given in \cite{Cui2021} to $(g_{-1})_{\ast}E_{\alpha}$, note the second fundamental forms is also equivariant, then $\langle h((g_{-1})_{\ast}E_{\alpha},(g_{-1})_{\ast}E_{\beta}),\xi_{1}\rangle=\langle h(E_{\alpha},E_{\beta}),g_{\ast}\xi_{1}\rangle$, the computation is simplified), $|\xi_{1}|^2=a^2$, note the image lies in a sphere of radius $\frac{1}{\sqrt{3}}$, then the value of $H_{1}^{\xi_{1}}$ at the sphere of radius $1$ is given by $\pm\frac{a}{\sqrt{3}}=\pm\frac{|\xi_{1}|}{\sqrt{3}}$.

Follow \eqref{base equation},
\begin{equation}\label{base equation two}
I-tH^{v}=diag\{ (1-\frac{b_{1}t}{\lambda_{1}})I_{k_{1}}-\frac{t}{\lambda_{1}}H_{1}^{\xi_{1}},\ldots,
(1-\frac{b_{m}t}{\lambda_{m}})I_{k_{m}}-\frac{t}{\lambda_{m}}H_{m}^{\xi_{m}}\},
\end{equation}
hence
\begin{equation}
det(I-tH^{v}_{ij})=\prod_{i=1}^{3}((1-\sqrt{3}b_{i}t)^2-|\xi_{i}|^2t^2),
\end{equation}
where $\sum_{i=1}^{3}(b_{i}^2+|\xi_{i}|^2)=1$.

In the next, $t$ is restricted on $(0,\frac{1}{\sqrt{5}})$, this choice of $t$ is enough for ensuring that the normal wedge of angle radius $arctant$ won't meets the focal points of the cone. Denote $|\xi_{i}|^2=c_{i}\geq 0$, we claim that if minimum value of $det(I-tH^{v}_{ij})$ is attained for fixed $b_{i},c_{i}(1\leq i \leq 3)$, then no two of $c_{1},c_{2},c_{3}$ are non-zeros. If so, assume $c_{2}>0,c_{3}>0$, then for fixed $b_{1},b_{2},b_{3},c_{1}$, the sum of $c_{2}$ and $c_{3}$ is constant. We adjust $c_{2},c_{3}$ as follows: if $(1-\sqrt{3}b_{2}t)^2\geq(1-\sqrt{3}b_{3}t)^2$, i.e. $b_{2}\leq b_{3}$, then we let $c_{2}$ be its possible maximal value and $c_{3}$ be zero, the result will be more small, this is an contradiction.

So, we can assume $c_{2}=c_{3}=0$, then
\begin{equation}
det(I-tH^{v}_{ij})=\left((1-\sqrt{3}b_{1}t)^2-c_{1}t^2\right)(1-\sqrt{3}b_{2}t)^2(1-\sqrt{3}b_{3}t)^2,
\end{equation}
subject to
\begin{equation}
\begin{cases}
b_{1}+b_{2}+b_{3}=0, \\
b_{1}^2+b_{2}^2+b_{3}^2+c_{1}=1,c_{1}\geq 0.
\end{cases}
\end{equation}

We express $det(I-tH^{v}_{ij})$ as functions of $t,b_{1},c_{1}$, for convenience, we replace $b_{1}$ by $a$, $c_{1}$ by $b^2$, the above equations should have two real roots for $b_{2},b_{3}$, then we can let the domain be
\begin{equation}
D\equiv\{(a,b)\in \mathbb{R}^2|3a^2+2b^2\leq 2 \},
\end{equation}
and
\begin{equation}
g(a,b):=det(I-tH^{v}_{ij})=((1-\sqrt{3}at)^2-b^2t^2)(1+\sqrt{3}at+(3a^2+\frac{3}{2}b^2-\frac{3}{2})t^2)^2
\end{equation}

Hence(we get almost all the following computations by using $Mathematica$),
\begin{equation}
\frac{\partial g}{\partial a}=-t^{3}(1+\sqrt{3}at+(3a^2+\frac{3}{2}b^2-\frac{3}{2})t^2)
(-3\sqrt{3}+18\sqrt{3}a^{2}+5\sqrt{3}b^{2}+9at-54a^{3}t+3ab^{2}t),
\end{equation}
and
\begin{equation}
\frac{\partial g}{\partial b}=-bt^{2}(1+\sqrt{3}at+(3a^2+\frac{3}{2}b^2-\frac{3}{2})t^2)
(-4+14\sqrt{3}at-3t^2-12a^{2}t^{2}+9b^{2}t^{2}).
\end{equation}

Consider $1+\sqrt{3}at+(3a^2+\frac{3}{2}b^2-\frac{3}{2})t^2$ as a quadratic function of $a$ and compute the discriminant, we have $1+\sqrt{3}at+(3a^2+\frac{3}{2}b^2-\frac{3}{2})t^2>0$ on $0<t<\frac{1}{\sqrt{5}}$.

Solve the equations: $\frac{\partial g}{\partial a}=0$ and $\frac{\partial g}{\partial b}=0$, we conclude that in the interior of $D$, i.e. $3a^2+2b^2<2$,
$g(a,b)$ attains minimum value at $a = -\frac{1}{6}$, $b = 0$(another critical point is $a = -\frac{1}{6}$, $b = 0$ which is bigger than this one). Note that the norm of the second fundamental form attains maximum value $\sqrt{6}$ when $(a, b) = (-\frac{1}{6}, 0)$ or $(\sqrt{\frac{2}{3}}, 0)$, and they own the same expression of $det(I-tH^{v}_{ij})$.

On boundary of $D$,
\begin{equation}\label{boundary equation}
f(a):=g\left(a,\sqrt{1-\frac{3}{2}a^2}\right)=-\frac{1}{288}(2\sqrt{3}+3at)^4(-2+4\sqrt{3}at+2t^2-9a^2t^2),
\end{equation}
subject to $a^2\leq \frac{2}{3}$, and
\begin{equation}
f'(a)=\frac{1}{48}t^2(2\sqrt{3}+3at)^3(27ta^2-4\sqrt{3}a-4t).
\end{equation}

By analyzing the quadratic function $27ta^2-4\sqrt{3}a-4t$ in the interval $a^2\leq \frac{2}{3}$, we find that if and only if $0<t<\frac{2\sqrt{2}}{7}$, the sign of $f'(a)$ is $+,-$(if $t>\frac{2\sqrt{2}}{7}$, the sign is $+,-,+$). Additionally, we compute that: $g(-\sqrt{\frac{2}{3}})>g(\sqrt{\frac{2}{3}})$, then $g(\sqrt{\frac{2}{3}})$ is the minimum value on the boundary.

The vanishing angle associated to $E(t)$ is $19.9^{\circ}$, and $tan(19.9^{\circ})< \frac{2\sqrt{2}}{7}$, since the normal radius is $arccos(\frac{1}{2})=60^{\circ}$ by proposition \ref{normal radius}, hence

\begin{theorem}
The cone over products of three Veronese embeddings $\mathbb{R}P^2 \times \mathbb{R}P^2  \times \mathbb{R}P^2$ is area-minimizing.
\end{theorem}

Moreover, it is easy to see that the cases for $S^2 \times S^2  \times \mathbb{R}P^2$ and $S^2 \times \mathbb{R}P^2 \times \mathbb{R}P^2$ are attributed to the above discussions, by a further checking for normal radius, we also have

\begin{theorem}
The cones over $S^2 \times S^2  \times \mathbb{R}P^2$, $S^2 \times \mathbb{R}P^2 \times \mathbb{R}P^2$ are area-minimizing.
\end{theorem}

(2)~~$\mathbb{R}P^2 \times G(2,4;\mathbb{R})$($S^2 \times G(2,4;\mathbb{R})$):

This case is similar to $\mathbb{R}P^2 \times \mathbb{R}P^2 \times \mathbb{R}P^2$, we will prove on $[0,\frac{2\sqrt{2}}{7})$, for any chosen unit normal vector $v$,
\begin{equation}
inf_{v}det(I-tH^{v}_{ij})=(1-t\sqrt{2})^2\left(1+\frac{t}{\sqrt{2}}\right)^4=E(t).
\end{equation}

First, $G(2,4;\mathbb{R})$ is embedded in $H(4,\mathbb{R})$(or its hyperplane of trace zero), the normal vector $\xi_{2}$ can be given by $diag\{a,-a,b,-b\}$(to see this, first assume $\xi_{2}$ is an any given normal vector, if $g_{\ast}\xi_{2}:=g\xi_{2}g^{T}$ is diagonal for an isotropy isometry action $g$, we could change the standard orthonormal tangent basis $E_{a}^{\alpha}:=E_{\alpha a}+E_{a \alpha}(1\leq a \leq 2, 3\leq \alpha \leq 4)$ given in \cite{Cui2021} to $(g_{-1})_{\ast}E_{a}^{\alpha}$, note the second fundamental forms is also equivariant, then the computation is simplified). The result $H_{2}^{\xi_{2}}$ has value $diag\{-a+b,a+b,-a-b,a-b\}$. Note the image lies in a sphere of radius $\frac{1}{\sqrt{2}}$, then the value of $H_{2}^{\xi_{2}}$ at the sphere of radius $1$ is given by $\frac{1}{\sqrt{2}}diag\{-a+b,a+b,-a-b,a-b\}$, we let $x=a-b,y=a+b$. The discussion for the first factor $\mathbb{R}P^2$ is the same to $\mathbb{R}P^2 \times \mathbb{R}P^2 \times \mathbb{R}P^2$, we assume $|\xi_{1}|=c$, since $b_{1}+\sqrt{2}b_{2}=0$, then

\begin{equation}
\begin{aligned}
det(I-tH^{v}_{ij})&=((1-\sqrt{3}b_{1}t)^2-c^2t^2)\left((1+\frac{\sqrt{3}}{2}b_{1}t)^2-\frac{3}{4}x^2t^2\right)\\
&\times \left((1+\frac{\sqrt{3}}{2}b_{1}t)^2-\frac{3}{4}y^2t^2\right),
\end{aligned}
\end{equation}
it subject to
\begin{equation}
\frac{3}{2}b_{1}^2+c^2+\frac{1}{2}(x^2+y^2)=1.
\end{equation}

Fix $b_{1}$, let $A=(1-\sqrt{3}b_{1}t)^2$, $B=(1+\frac{\sqrt{3}}{2}b_{1}t)^2$, $\beta_{1}=c^2$, $\beta_{2}=x^2$, $\beta_{3}=y^2$, then we want to get the minimum value of
\begin{equation}
f(\beta_{1},\beta_{2},\beta_{3})=(A-t^2\beta_{1})\left(B-\frac{3}{4}t^2\beta_{2}\right)\left(B-\frac{3}{4}t^2\beta_{3}\right),
\end{equation}
subject to
$D=\{(\beta_{1},\beta_{2},\beta_{3})\in \mathbb{R}^3| \beta_{1}+\frac{1}{2}\beta_{2}+\frac{1}{2}\beta_{3}=1-\frac{3}{2}b_{1}^2, \beta_{i}\geq 0(i=1,2,3)\}$.

We can use the method of Lagrange Multiplier to show that there are no minimum points in the interior of $D$, let $g(\beta_{1},\beta_{2},\beta_{3})=\beta_{1}+\frac{1}{2}\beta_{2}+\frac{1}{2}\beta_{3}-c$, where $c$ is a positive number. The gradient of $g$: $\nabla g=(1,\frac{1}{2},\frac{1}{2})$, the gradient of $f$ is given by $\nabla f=(\frac{\partial f}{\partial \beta_{1}},\frac{\partial f}{\partial \beta_{2}},\frac{\partial f}{\partial \beta_{3}})$, where
\begin{equation}
\begin{cases}
\frac{\partial f}{\partial \beta_{1}}&=-t^2(B-\frac{3}{4}t^2\beta_{2})(B-\frac{3}{4}t^2\beta_{3})<0, \\
\frac{\partial f}{\partial \beta_{2}}&=-\frac{3}{4}t^2(A-t^2\beta_{1})(B-\frac{3}{4}t^2\beta_{3})<0, \\
\frac{\partial f}{\partial \beta_{3}}&=-\frac{3}{4}t^2(A-t^2\beta_{1})(B-\frac{3}{4}t^2\beta_{2})<0.
\end{cases}
\end{equation}

At a critical point $(\beta_{1},\beta_{2},\beta_{3})$,
\begin{equation}
\nabla f // \nabla g \Leftrightarrow \frac{\partial f}{\partial \beta_{2}}=\frac{\partial f}{\partial \beta_{3}}=\frac{1}{2}\frac{\partial f}{\partial \beta_{1}},
\end{equation}
then $\beta_{2}=\beta_{3}$, $B-\frac{3}{4}t^2\beta_{2}=B-\frac{3}{4}t^2\beta_{3}=:\lambda>0$, $A-t^2\beta_{1}=\frac{2}{3}(B-\frac{3}{4}t^2\beta_{2})=\frac{2\lambda}{3}$.

The Hessian matrix $Hess(f)$ is of the form
\begin{equation}
\begin{pmatrix}
0 &c &c \\
c &0 &\frac{c}{2} \\
c &\frac{c}{2} &0
\end{pmatrix},
\end{equation}
where $c=\frac{3t^4\lambda}{4}$.

A tangent vector $v$ of the algebraic manifold $g=0$ is given by $\{v=(x,y,z)\in \mathbb{R}^{3}|2x+y+z=0\}$, then the quadratic form $vHess(f)v^{T}=-c(y^2+yz+z^2)<0$ shows that there are no minimum points in the interior of $D$.

On boundary $\partial D$, assume $\beta_{1}=0$, then
\begin{equation}
\frac{f}{A}=B^2-\frac{3}{4}t^2(\beta_{2}+\beta_{3})+\frac{9}{16}t^4\beta_{2}\beta_{3},
\end{equation}
attains minimum when at least one of $\beta_{2},\beta_{3}$ is zero.

assume $\beta_{2}=0$, then
\begin{equation}
\frac{f}{B}=AB-t^2\left(B\beta_{1}+\frac{3A}{2}\frac{\beta_{3}}{2}\right)+\frac{3}{4}t^4\beta_{1}\beta_{3},
\end{equation}
attains minimum when $\beta_{3}=0$ if $B\geq \frac{3A}{2}$, when $\beta_{1}=0$ if $B\leq \frac{3A}{2}$.

So,  $f(\beta_{1},\beta_{2},\beta_{3})$ attains minimum when at least two of $\beta_{1},\beta_{2},\beta_{3}$ are zeros.

~($a$): $\beta_{2}=\beta_{3}=0$,
\begin{equation}
det(I-tH^{v}_{ij})=\left(1-2\sqrt{3}b_{1}t+\frac{9}{2}b_{1}^2t^2-t^2\right)\left(1+\frac{\sqrt{3}}{2}b_{1}t\right)^4,
\end{equation}
and $b_{1}^2\leq \frac{2}{3}$.

This equation is just \eqref{boundary equation}, so on $0<t<\frac{2\sqrt{2}}{7}$, the minimum is $E(t)$ and it is attained when $b_{1}=\sqrt{\frac{2}{3}}$, $c=0$.

~($b$): $\beta_{1}=\beta_{3}=0$, i.e. $c=x=0$, (the same for $\beta_{1}=\beta_{2}=0$),
\begin{equation}
f=AB\left(B-\frac{3}{4}y^2t^2\right),
\end{equation}
subject to $\frac{3}{2}b_{1}^2+\frac{1}{2}y^2=1$, i.e. $y^2=2-3b_{1}^2$, then we need to find the minimum of
\begin{equation}
f(b_{1})=(1-\sqrt{3}b_{1}t)^2\left(1+\frac{\sqrt{3}}{2}b_{1}t\right)^2\left(1+\sqrt{3}b_{1}t
+3b_{1}^2t^2-\frac{3}{2}t^2\right),
\end{equation}
where $b_{1}^2\leq \frac{2}{3}$.

And
\begin{equation}
f'(b_{1})=\frac{3}{4}t^2(\sqrt{3} + 6 t b_{1}) (-1 +\sqrt{3} t b_{1}) (2 + \sqrt{3} t b_{1}) (-t +
\sqrt{3} b_{1}+ 3 t b_{1}^2).
\end{equation}

When $0<t<\frac{\sqrt{2}}{4}$, the sign of $f'(b_{1})$ is $+,-$, we need to compare $f(-\sqrt{\frac{2}{3}})$ and $f(\sqrt{\frac{2}{3}})$, the minimum point is $f(\sqrt{\frac{2}{3}})=E(t)$.

When $\frac{\sqrt{2}}{4}<t<\frac{1}{\sqrt{5}}$, the sign of $f'(b_{1})$ is $-,+,-$(we note here it is a little complicated, the signs of $\sqrt{3} + 6 t b_{1}$ and $-t +
\sqrt{3} b_{1}+ 3 t b_{1}^2$ are all depending on $t$), the smaller interior critical point is $b_{1}=x_{0}=-\frac{1}{2\sqrt{3}t}$, we compute that
\begin{equation}
f(x_{0})-f(\sqrt{\frac{2}{3}})=\frac{(\sqrt{2} - 2 t)(\sqrt{2} + 4 t)^2 (-13 \sqrt{2} + 78 t +
   48\sqrt{2} t^2 + 16 t^3)}{1024},
\end{equation}
it is positive.

Hence, for case($b$), when $0<t<\frac{1}{\sqrt{5}}$, $f(b_{1})$ attains the minimum $E(t)$ when $b_{1}=\sqrt{\frac{2}{3}}$.

In summary, when $0<t<\frac{2\sqrt{2}}{7}$, the minimum of $det(I-tH^{v}_{ij})$ is $E(t)$. The vanishing angle associated to $E(t)$ is $19.9^{\circ}$, and $tan(19.9^{\circ})\leq \frac{2\sqrt{2}}{7}$,  the normal radius is $arccos(\frac{1}{2})=60^{\circ}$, then we have

\begin{theorem}
The cone over $\mathbb{R}P^2 \times G(2,4;\mathbb{R})$ is area-minimizing.
\end{theorem}

Moreover, it is easy to see that the case for $\mathbb{R}P^2 \times G(2,4;\mathbb{R})$ is attributed to the above discussions,  by a further checking for normal radius, we also have

\begin{theorem}
The cone over $\mathbb{R}P^2 \times G(2,4;\mathbb{R})$ is area-minimizing.
\end{theorem}

(3)~~$\mathbb{R}P^2 \times \mathbb{R}P^4$($S^2 \times \mathbb{R}P^4$,$S^4 \times \mathbb{R}P^2$):

$\mathbb{R}P^4$ is embedded in $H(5,\mathbb{R})$, the normal vector $\xi_{2}$ can be given by $diag\{0,c_{1},c_{2},c_{3},c_{4}\}$ where $\sum_{i=1}^4c_{i}=0$. Note the image lies in a sphere of radius $\frac{\sqrt{2}}{\sqrt{5}}$, then the value of $H_{2}^{\xi_{2}}$ at the sphere of radius $1$ is given by $\frac{\sqrt{2}}{\sqrt{5}}diag\{c_{1},c_{2},
c_{3},c_{4}\}$. Denote $|\xi_{1}|=d$, let $\alpha^2= \sum_{i=1}^{4} c_{i}^2$, then $|\xi_{2}|^2=\frac{\alpha^2}{2}$, hence $\frac{\alpha^2}{2}+\frac{3b_{1}^2}{2}+d^2=1$. Since $\lambda_{1}=\frac{1}{\sqrt{3}}$, so $b_{1}+\sqrt{2}b_{2}=0$, and
\begin{equation}
det(I-tH^{v}_{ij})=((1-\sqrt{3}b_{1}t)^2-d^2t^2)\prod_{i=1}^{4}\left(1+
\frac{\sqrt{3}}{2}b_{1}t-\frac{\sqrt{3}}{\sqrt{5}}c_{i}t\right).
\end{equation}

Follow \cite{lawlor1991sufficient}, for fixed $b_{1},d$, on $t\in(0, \frac{1}{\sqrt{5}})$, the minimum of $\prod_{i=1}^{4}(1+
\frac{\sqrt{3}}{2}b_{1}t-\frac{\sqrt{3}}{\sqrt{5}}c_{i}t)$ is attained when $c_{2}=c_{3}=c_{4}=-\frac{c_{1}}{3}$ are all negative, and it is $(1+
\frac{\sqrt{3}}{2}b_{1}t-\frac{\sqrt{3}}{\sqrt{5}}c_{1}t)(1+
\frac{\sqrt{3}}{2}b_{1}t+\frac{1}{\sqrt{15}}c_{1}t)^3$, for convenience, we let $b$ be $\frac{\sqrt{2}}{\sqrt{3}}c_{1}$, let $a$ be $b_{1}$, then  we need to find the minimum of
\begin{equation}
\begin{aligned}
g(a,b)&=\left((1-\sqrt{3}at)^2-(1-\frac{3}{2}a^2-b^2)t^2\right) \\
&\times \left(1+\frac{\sqrt{3}}{2}at-\frac{3}{\sqrt{10}}bt\right)
\left(1+\frac{\sqrt{3}}{2}at+\frac{1}{\sqrt{10}}bt\right)^3
\end{aligned}
\end{equation}
in the domain $D=\{(a,b)\in \mathbb{R}^2|3a^2+2b^2\leq2, b\geq 0\}$.

There are no critical points in the interior of $D$ when $t$ restricted on $(0,\frac{1}{\sqrt{5}})$ by using $Mathematica$.

The boundary of $D$ are divided into two parts,

~($a$): $b = 0, a^2\leq \frac{2}{3}$,
\begin{equation}
g(a,0)=\frac{1}{32}(2 + \sqrt{3}at)^4(2 - 4 \sqrt{3}at-2t^2+9a^2t^2),
\end{equation}
this is just the function \ref{boundary equation} for the boundary case of $\mathbb{R}P^2 \times \mathbb{R}P^2 \times \mathbb{R}P^2$, so for this case, when $0<t<\frac{2\sqrt{2}}{7}$, the minimum is $E(t)$.

~($b$): on the ellipse $b=\sqrt{1-\frac{3}{2}a^2}$, define $f(a)=g(a,\sqrt{1-\frac{3}{2}a^2})$, then
\begin{equation}
\begin{aligned}
f(a)&=\frac{(1-\sqrt{3}at)^2}{10000} (10+5\sqrt{3}at-3\sqrt{5}\sqrt{2 - 3 a^2}t)\\
&\times(10+5\sqrt{3}at+ \sqrt{5}\sqrt{2 - 3 a^2}t)^3,
\end{aligned}
\end{equation}
and
\begin{equation}
\begin{aligned}
f'(a)&=\frac{3t^2(\sqrt{3}at-1)}{500} (10+5\sqrt{3}at+\sqrt{5}\sqrt{2-3a^2}t)^2 \\
&\times (12a -\sqrt{3}t+ 12\sqrt{3}a^2t-3\sqrt{5}a\sqrt{2-3a^2}t).
\end{aligned}
\end{equation}

Let $h(a)=12a -\sqrt{3}t+ 12\sqrt{3}a^2t-3\sqrt{5}a\sqrt{2-3a^2}t$, we will show that $h$ is increasing in $a$. Since in the following we can see $h'(a)\rightarrow +\infty$ when $a\rightarrow \pm \sqrt{\frac{2}{3}}$, it suffices to prove that the minimum value of $h'(a)$ is no less than zero.

By using $Mathematica$,
\begin{equation}
h'(a)=12+24\sqrt{3}at+\frac{9\sqrt{5}a^2t}{\sqrt{2 - 3 a^2}}-3\sqrt{5}\sqrt{2 - 3 a^2}t,
\end{equation}
and
\begin{equation}
h''(a)=\frac{6t(9\sqrt{5}a-9\sqrt{5}a^3+4\sqrt{3}(2-3a^2)^{\frac{3}{2}})}{(2-3a^2)^{\frac{3}{2}}}.
\end{equation}

It happens that the minimum points of $h'(a)$, i.e. the roots of $h''(a)$, is independent of $t$.

The numerical solution of $h''(a)=0$ is given by:
\begin{equation}
{a\rightarrow -0.555087}, {a \rightarrow 0.922469 - 0.070777i}, {a \rightarrow 0.922469 + 0.070777i}.
\end{equation}

the unique real root $x_{0}=-0.555087$ is the minimum point of $h'(a)$. When $a\rightarrow x_{0}$, the minimum value of $h'(a)$ is $12-24.053t$ which is decreasing in $t$, and when $t=\frac{2\sqrt{2}}{7}$, it is $2.28114$, then, we have proved that $h(a)$ is increasing in $a$.

Additionally, $h(-\sqrt{\frac{2}{3}})=7\sqrt{3}t-4\sqrt{6}<0$ and
$h(\sqrt{\frac{2}{3}})=7\sqrt{3}t+4\sqrt{6}>0$, it tells us that the sign of $f'(a)$ is $+,-$.

Therefore the minimum value of $f(a)$ is the smaller of $f(-\sqrt{\frac{2}{3}})$ and $f(-\sqrt{\frac{2}{3}})$, which is $f(-\sqrt{\frac{2}{3}})=E(t)$ if $0 < t <\frac{2\sqrt{2}}{7}$.

The vanishing angle associated to $E(t)$ is $19.9^{\circ}$, and $tan(19.9^{\circ})\leq \frac{2\sqrt{2}}{7}$, the normal radius is $arccos(\frac{1}{2})=60^{\circ}$, hence

\begin{theorem}
The cone over $\mathbb{R}P^2 \times \mathbb{R}P^4$ is area-minimizing.
\end{theorem}

Moreover, it is easy to see that the cases for $S^2 \times \mathbb{R}P^4$ and $S^4 \times \mathbb{R}P^2$ are attributed to the above discussions, by a further checking for normal radius, we also have

\begin{theorem}
The cones over $S^2 \times \mathbb{R}P^4$, $S^4 \times \mathbb{R}P^2$ are area-minimizing.
\end{theorem}

(4)~~$\mathbb{R}P^3 \times \mathbb{R}P^3$($S^3 \times \mathbb{R}P^3$):

$\mathbb{R}P^3$ is embedded in $H(4,\mathbb{R})$, the normal vector $\xi_{1}$ can be given by $diag\{0,c_{1},c_{2},c_{3}\}$, respectively, $\xi_{2}=diag\{0,d_{1},d_{2},d_{3}\}$, where $\sum_{i=1}^3c_{i}=0$ and $\sum_{i=1}^3d_{i}=0$. Note the image lies in a sphere of radius $\frac{\sqrt{3}}{2\sqrt{2}}$, then the value of $H_{1}^{\xi_{1}}$ at the sphere of radius $1$ is given by $\frac{\sqrt{3}}{2\sqrt{2}}diag\{c_{1},c_{2},c_{3}\}$, the similar results for $H_{2}^{\xi_{2}}$. $|\xi_{1}|^2=\frac{1}{2}\sum_{i=1}^3c_{i}^2$, $|\xi_{2}|^2=\frac{1}{2}\sum_{i=1}^3d_{i}^2$. Now $\lambda_{1}=\lambda_{2}=\frac{1}{\sqrt{2}}$, then $b_{1}+b_{2}=0$, $\frac{\sum_{i=1}^3c_{i}^2}{2}+\frac{\sum_{i=1}^3d_{i}^2}{2}+2b_{1}^2=1$, and
\begin{equation}
det(I-tH^{v}_{ij})=\prod_{i=1}^{3}\left(1-\sqrt{2}b_{1}t-\frac{\sqrt{3}}{2}c_{i}t\right)\left(1+\sqrt{2}b_{1}t-\frac{\sqrt{3}}{2}d_{i}t\right).
\end{equation}

For fixed $b_{1},d_{i}(i=1,2,3)$, the minimum of $\prod_{i=1}^{3}(1-\sqrt{2}b_{1}t-\frac{\sqrt{3}}{2}c_{i}t)$ is attained when $c_{1}$ is positive and $c_{2}=c_{3}=-\frac{c_{1}}{2}$, similar for the case considering $b_{1},c_{i}(i=1,2,3)$ fixed, all these can attained on $t\in (0,\frac{1}{\sqrt{5}})$(\cite{lawlor1991sufficient}). Then we can consider these normal vectors $v$ given as follows: let $v=(v_{1},v_{2})$, $v_{1}=b_{1}x_{1}+\xi_{1}$, $v_{2}=-b_{1}x_{2}+\xi_{2}$, where $x_{1},x_{2}$ is the position vectors, $\xi_{1}=diag\{0,c_{1},-\frac{c_{1}}{2},-\frac{c_{1}}{2}\}$, $\xi_{2}=diag\{0,d_{1},-\frac{d_{1}}{2},-\frac{d_{1}}{2}\}$, and $c_{1},d_{1}\geq 0$, $8b_{1}^2+3c_{1}^2+3d_{1}^2=4$. For convenience, we set $b_{1}=a$, $c_{1}=b\geq 0$, $d_{1}=c\geq 0$, then
\begin{align}
det(a,b,c):=det(I-tH^{v}_{ij})&=\left(1-\sqrt{2}at-\frac{\sqrt{3}}{2}bt\right)
\left(1-\sqrt{2}at+\frac{\sqrt{3}}{4}bt\right)^2 \\ \notag
&\times\left(1+\sqrt{2}at-\frac{\sqrt{3}}{2}ct\right)
\left(1+\sqrt{2}at+\frac{\sqrt{3}}{4}ct\right)^2,
\end{align}
and $8a^2+3b^2+3c^2=4, b\geq0, c\geq0$.

Since $b,c$ are in the symmetric positions, we can further assume $a\geq0$.

We let $f(a,b)=det(a,b,\sqrt{-\frac{8 a^2}{3}-b^2+\frac{4}{3}})$. First, we fix $a$ and determine the sign of $f'(b)$, here $ a\geq0, 0\leq b \leq \sqrt{-\frac{8 a^2}{3}+\frac{4}{3}}$.

By using $Mathematica$,
\begin{equation}\label{first order}
\begin{aligned}
f'(b)&=\frac{9}{512} b t^3 \left(-4 \sqrt{2} a t+\sqrt{3} b t+4\right) \left(t \sqrt{-8 a^2-3 b^2+4}+4 \sqrt{2} a t+4\right) \\
&\times (\sqrt{2} a t \sqrt{-8 a^2-3 b^2+4}+\sqrt{-8 a^2-3 b^2+4} \\
&-4 a^2 t+\sqrt{6} a b t-16 \sqrt{2} a-3 b^2 t-\sqrt{3} b+2 t )
\end{aligned}
\end{equation}
The last factor in \ref{first order} is denoted by $g(a,b)$, then easy to see it has the same sign of $f'(b)$.

Now
\begin{equation}
g\left(a,\sqrt{\frac{4}{3}-\frac{8 a^2}{3}}\right)=2 \left(2 a^2 t+\sqrt{2} \sqrt{1-2 a^2} a t-\sqrt{1-2 a^2}-8 \sqrt{2} a-t\right),
\end{equation}
it is less than zero, and
\begin{equation}
g'(b)=-\frac{3 \sqrt{2} a b t}{\sqrt{-8 a^2-3 b^2+4}}-\frac{3 b}{\sqrt{-8 a^2-3 b^2+4}}+\sqrt{6} a t-6 b t-\sqrt{3}
\end{equation}
is also less than zero.

So, the possible sign of $g(b)$(hence that of $f'(b)$) is $+,-$ or $-$, we compute that
\begin{equation}
\begin{aligned}
f(a,0)-f\left(a,\sqrt{\frac{4}{3}-\frac{8 a^2}{3}}\right)&=\frac{a \left(1-2 a^2\right) t^3}{\sqrt{2}}\\
&\times\left(2 a^2 \sqrt{1-2 a^2} t^3-12 a^2 t^2+3 \sqrt{1-2 a^2} t+6\right),
\end{aligned}
\end{equation}
it is obviously no less than zero which shows that when $a$ is fixed, $f(a,b)$ attains minimum value $h(a):=g(\sqrt{\frac{4}{3} - \frac{8}{3}a^2})$(i.e. $c$ is zero).

We compute that
\begin{equation}
\begin{aligned}
h'(a)&=\frac{3}{4} t^2 \left(\sqrt{2} a t+1\right)^2 \left(\sqrt{1-2 a^2} t-2 \sqrt{2} a t+2\right) \\
&\times \left(12 \sqrt{2} a^2 t+4 \sqrt{1-2 a^2} a t-6 a-\sqrt{2} t\right).
\end{aligned}
\end{equation}

Let $q(a):=4 a \sqrt{1-2 a^2} t-\left(-12 \sqrt{2} a^2 t+6 a+\sqrt{2} t\right)$, $4 a \sqrt{1-2 a^2} t\leq \sqrt{2} t$, where the equal sign holds iff $a = \frac{1}{2}$. Next, consider $r(a)=-12 \sqrt{2} a^2 t+6 a+\sqrt{2} t$ as a quadratic function of $a$, then its minimum is $r(0)=\sqrt{2}t$. Therefore $q (a) < 0$, and consequently $h' (a) < 0$.

$h' (a) < 0$, $h$ is decreasing in $a$. Therefore
\begin{equation}
\begin{aligned}
min (h)&= min (f) = h\left(\frac{1}{\sqrt{2}}\right) = f\left(\frac{1}{\sqrt{2}},0\right)\\
&= det\left(\frac{1}{\sqrt{2}}, 0, 0\right)=(1+t)^3(1-t)^3=G(t)
\end{aligned}
\end{equation}
when $t\in (0,\frac{1}{\sqrt{5}})$.

The vanishing angle associated to $G(t)$ is $19^{\circ}$, and $tan(19^{\circ})< \frac{1}{\sqrt{5}}$, the normal radius is $arccos(\frac{1}{3})>70^{\circ}$, then we have

\begin{theorem}
The cone over $\mathbb{R}P^3 \times \mathbb{R}P^3$ is area-minimizing.
\end{theorem}

Moreover, it is easy to see that the case for $S^3 \times \mathbb{R}P^3$ is attributed to the above discussions, by a further checking for normal radius, we also have

\begin{theorem}
The cone over $S^3 \times \mathbb{R}P^3$ is area-minimizing.
\end{theorem}

(5)~~$\mathbb{R}P^2 \times \mathbb{C}P^2$:

$\mathbb{R}P^2$ is embedded in $H(3,\mathbb{R})$, the normal vector $\xi_{1}$ can be given by $diag\{0,c,-c\}$, then $|\xi_{1}|^2=c^2$, the image lies in a sphere of radius $\frac{1}{\sqrt{3}}$, then the value of $H_{1}^{\xi_{1}}$ at the sphere of radius $1$ is given by $diag\{c,-c\}$.

$\mathbb{C}P^2$ is embedded in $H(3,\mathbb{C})$, the normal vector $\xi_{2}$ can be given by $diag\{0,b,-b\}$, then $|\xi_{2}|^2=b^2$. Note the image lies in a sphere of radius $\frac{1}{\sqrt{3}}$, then the value of $H_{2}^{\xi_{2}}$ at the sphere of radius $1$ is given by $\frac{1}{\sqrt{3}}diag\{b,b,-b,-b\}$(\cite{Cui2021}).

Now $\lambda_{1}=\frac{1}{\sqrt{3}}$, $\lambda_{2}=\frac{\sqrt{2}}{\sqrt{3}}$, then $b_{1}+\sqrt{2}b_{2}=0$, $b^2+c^2+\frac{3}{2}b_{1}^2=1$, and
\begin{equation}
det(I-tH^{v}_{ij})=\left((1-\sqrt{3}b_{1}t)^2-c^2t^2\right)\left(\left(1+\frac{\sqrt{3}}{2}b_{1}t\right)^2
-\frac{1}{2}b^2t^2\right)^2.
\end{equation}

Let $b_{1}$ be $a$, let $D=\{(a,b)\in \mathbb{R}^2|3a^2+2b^2\leq 2\}$, denote
\begin{equation}
g(a,b)=\left((1-\sqrt{3}at)^2-\left(1-\frac{3}{2}a^2-b^2\right)t^2\right)\left(\left(1+\frac{\sqrt{3}}{2}at\right)^2
-\frac{1}{2}b^2t^2\right)^2.
\end{equation}

There are no critical points in the interior of $D$ when $t$ restricted on $(0,\frac{1}{\sqrt{5}})$ by using $Mathematica$.

On boundary of $D$, i.e. $c=0$, let
\begin{equation}
\begin{aligned}
f(a):=g\left(a,\pm \sqrt{1-\frac{3 a^2}{2}}\right)&=(1-\sqrt{3} a t)^2 \left(1+\frac{\sqrt{3}}{2}at+\frac{bt}{\sqrt{2}}\right)^2 \\
&\times \left(1+\frac{\sqrt{3}}{2}at-\frac{bt}{\sqrt{2}}\right)^2,
\end{aligned}
\end{equation}
where $\frac{3}{2}a^2 +b^2=1$.

Let \begin{equation}
\begin{aligned}
\alpha^2 & \equiv \left(\frac{b_{1}}{\lambda_{1}}\right)^2+\left(\frac{\sqrt{3}}{2}a+\frac{b}{\sqrt{2}}\right)^2+\left(\frac{\sqrt{3}}{2}a
-\frac{b}{\sqrt{2}}\right)^2 \\
&= 3\left(\frac{3}{2}a^2+b^2\right)-2b^2 \\
&\leq 3.
\end{aligned}
\end{equation}

So, follow Lemma 3.2 and the Appendix in \cite{lawlor1991sufficient}, on $(0,\frac{1}{\sqrt{5}})\subset(0,\frac{1}{\sqrt{2}}]\subset(0,\frac{1}{\alpha}\sqrt {\frac{3}{2}}]$, we have
\begin{equation}
f(a)\geq F^2(\alpha,t,3) \geq F^2(\sqrt{3},t,3)=(1-t\sqrt{2})^2\left(1+\frac{t}{\sqrt{2}}\right)^4=E(t),
\end{equation}
the minimum $E(t)$ of $det(I-tH^{v}_{ij})$ can be attained by setting $b=0,c=0$, and $a=\sqrt{\frac{2}{3}}$.

The vanishing angle is $19.9^{\circ}$, and $tan(19.9^{\circ})< \frac{1}{\sqrt{5}}$, since the normal radius is $arccos(\frac{1}{2})=60^{\circ}$, hence

\begin{theorem}
The cone over $\mathbb{R}P^2 \times \mathbb{C}P^2$ is area-minimizing.
\end{theorem}

Moreover, it is easy to see that the case for $S^2 \times \mathbb{C}P^2$ is attributed to the above discussions, by a further checking for normal radius, we also have

\begin{theorem}
The cone over $S^2 \times \mathbb{C}P^2$ is area-minimizing.
\end{theorem}

\begin{remark}
A special class of cones we note here is, if the cone is not the product of spheres and one of the factors are $S^1$($\cong\mathbb{R}P^1$), then the minimum value of Jacobian $det(I-tH^{v}_{ij})$ attains $F(t)=(1-t\sqrt{5})(1+\frac{t}{\sqrt{5}})^{5}$. Then we can only conclude that these cones are stable, different like the case of product of spheres, the standard embedding of a Grassmannian into sphere is an isolated orbit of some polar group action(i.e. it is not a principal orbit of polar group action talked by Lawlor\cite{lawlor1991sufficient}), then it's still of possibility for these cones being area-minimizing.
\end{remark}

\section{The $Pl\ddot{u}cker \ embedding$ and cones over products of oriented real Grassmannians}
In the last part in \cite{Cui2021}, we have studied the $Pl\ddot{u}cker \ embedding$ of oriented real Grassmannian $\widetilde{G}(l,k;\mathbb{R})$ into unit sphere of exterior vector space, then their cones are proved area-minimizing except $\widetilde{G}(2,4;\mathbb{R})$, in this section, we prove the cone over minimal product of them is also area-minimizing.

First, we show that the $Pl\ddot{u}cker \ embedding$ of oriented real Grassmannian $\widetilde{G}(2,2n+1;\mathbb{R})$(respectively $\widetilde{G}(2,2n;\mathbb{R})$) is equivalent to the isotropy representation of symmetric spaces of type $B$(respectively $D$), thus become symmetric $R-$spaces, see \cite{hirohashi2000area},\cite{kanno2002area}, then it naturally links to the work in \cite{Ohno2021area},\cite{tang2020minimizing}.

For $(G,K)=(SO(2n+1)^2,SO(2n+1))$-symmetric pairs of type $B_{n}$, the involution $\theta$ is given by $\theta(g_{1},g_{2})=(g_{2},g_{1})$ for $g_{1},g_{2}\in SO(2n+1)$. The Cartan decomposition $\mathfrak{g}=\mathfrak{l}\oplus \mathfrak{m}$ is given by: $\mathfrak{g}=\mathfrak{so}(2n+1)\times \mathfrak{so}(2n+1)$, $\mathfrak{l}=\{(x,x)|x\in \mathfrak{so}(2n+1)\}$, $\mathfrak{m}=\{(x,-x)|x\in \mathfrak{so}(2n+1)\}$.

\begin{prop}
the $Pl\ddot{u}cker \ embedding$ of $\widetilde{G}(2,2n+1;\mathbb{R})$ is equivalent to the isotropy representation of $(G,K)=(SO(2n+1)^2,SO(2n+1))$-symmetric spaces of type $B_{n}$.
\end{prop}

\begin{proof}
$\mathfrak{m}$ is isomorphic to $\mathfrak{so}(2n+1)$, under this identification, the adjoint action of $K$ on $\mathfrak{m}$ is given by: $Ad(k)\cdot x=kxk^{t}$ for $k \in K, x\in \mathfrak{so}(2n+1)$. Choose base point $x_{0}=G_{12}=E_{12}-E_{21}$, we have isotropy subgroup $Z_{K}^{x_{0}}=SO(2)\times SO(2n-1)$, then the orbit through $x_{0}$ is $SO(2n+1)/SO(2)\times SO(2n-1)\cong \widetilde{G}(2,2n+1;\mathbb{R})$.

Given $\mathfrak{so}(2n+1)$ the Euclidean metric: $g(A,B)=-\frac{1}{2}tr AB$ for $A,B\in\mathfrak{so}(2n+1)$, an orthonormal basis is given by: $G_{ij}=E_{ij}-E_{ji}(1\leq i<j\leq 2n+1)$.

$\wedge^{2}\mathbb{R}^{2n+1}$ is equipped with the induced Euclidean metric, then there is an isometric isomorphic:
\begin{equation}
\begin{aligned}
\wedge^{2}\mathbb{R}^{2n+1} &\cong \mathfrak{so}(2n+1) \\
e_{i}\wedge e_{j}&\leftrightarrow G_{ij},
\end{aligned}
\end{equation}
for $1\leq i<j\leq 2n+1$.

The orbit through $e_{1}\wedge e_{2}$ of $Pl\ddot{u}cker \ embedded$ $\widetilde{G}(2,2n+1;\mathbb{R})$ is $Ae_{1}\wedge Ae_{2}$, $A\in SO(2n+1)$, let $Ae_{i}=A_{ji}e_{j}$, the coordinate of this orbit is $Ae_{1}\wedge Ae_{2}=\Sigma_{i<j}(A_{i1}A_{j2}-A_{j1}A_{i2})e_{i}\wedge e_{j}$.

The base point $e_{1}\wedge e_{2}$ is corresponding to $G_{12}$, the orbit of adjoint representation associated to $\widetilde{G}(2,2n+1;\mathbb{R})$  is given by: $AG_{12}A^{t}=\Sigma_{i<j}(A_{i1}A_{j2}-A_{j1}A_{i2})G_{ij}$, then under the above identifications, these two orbits are the same.
\end{proof}

The case for $\widetilde{G}(2,2n;\mathbb{R})$ is similar.

Now, for each $1\leq i\leq m$, assume $f_{i}:M_{i}=\widetilde{G}(l_{i},k_{i};\mathbb{R})\hookrightarrow S^{a_{i}-1}(1)$ is an $Pl\ddot{u}cker \ embedding$, where $a_{i}=C_{k_{i}}^{l_{i}}$ is the combination number.

Consider the minimal product,
\begin{equation}
\begin{aligned}
f:M=\widetilde{G}(l_{1},k_{1};\mathbb{R})\times \cdots \times\widetilde{G}(l_{m},k_{m};\mathbb{R}) &\rightarrow S^{a_{1}+\cdots+ a_{m}-1}(1)\\
(x_{1},\ldots,x_{m})&\mapsto (\lambda_{1}f_{1}(x_{1}),\ldots,\lambda_{m}f_{m}(x_{m})),
\end{aligned}
\end{equation}
where each $f_{i}(x_{i})$ is written in the unit simple $l_{i}$-vectors, $\lambda_{i}=\sqrt{\frac{l_{i}(k_{i}-l_{i})}{dim \ M}}$, $dim\ M=\sum_{i=1}^{m}l_{i}(k_{i}-l_{i})$.

Denote the cone over $M$ by $C$, the upper bound of second fundamental forms of each $\widetilde{G}(l_{i},k_{i};\mathbb{R})$ is $\alpha_{i}^2=4$(\cite{Cui2021}), follow Theorem \ref{forms formula}, we have

\begin{prop}
The upper bound of the second fundamental forms of $C$ at the points belong to the unit sphere is given by $\alpha^2:=sup_{v}|H^{v}|^2=dim \ M$.
\end{prop}

For the normal radius of $C$,
we have
\begin{theorem}
The normal radius of $C$ is $arccos(1-\lambda_{1}^2)$, where $\lambda_{1}$ is arranged such that $\lambda_{1}=min\{\lambda_{1},\ldots,\lambda_{m}\}$.
\end{theorem}
\begin{proof}
For a fixed $i\in \{1,\ldots,m\}$, the nearest point to the origin $E_{i}:=e_{i,1}\wedge \cdots \wedge e_{i,l_{i}}$ of $M_{i}$ which belongs to the cone $C(M_{i})$ is shown to be one of the normal unit simple vectors in \cite{Cui2021}, denote it by $T_{i}$.

Then the nearest point to the origin $(\lambda_{1}E_{1},\cdots,\lambda_{m}E_{m})$ must be one of $P_{i}$, here $P_{i}$ are those points that the $i$-factors $\lambda_{i}E_{i}$ are replaced by $\lambda_{i}T_{i}$, the associated angle is \begin{equation}
cos \theta_{i}=\lambda_{1}^{2}+\cdots \lambda_{i-1}^{2}+\lambda_{i+1}^{2}+\cdots+\lambda_{m}^{2}=1-\lambda_{i}^{2}
\end{equation},
then the normal radius is the minimum value of $\theta_{i}(1\leq i\leq m)$.
\end{proof}

\begin{theorem}
Cones over minimal product of $Pl\ddot{u}cker \ embedded$, oriented real Grassmannians are area-minimizing.
\end{theorem}
\begin{proof}
There only exists one case for $7\leq dim M\leq 11$, it is $\widetilde{G}(2,4;\mathbb{R})\times \widetilde{G}(2,4;\mathbb{R})$, $(dim C,\alpha^2)=(9,8)$, then the estimated vanishing angle is $12.99^{\circ}$ by Lawlor's table, and the normal radius is $arccos(1-\frac{1}{2})=60^{\circ}$, it is area-minimizing;

For $dim M\geq 12$, let $k=dim C=dim M+1\geq 13$, we use the following formula given by Gary R. Lawlor:
\begin{equation}
tan(\theta_{2}(k,\alpha))< \frac{12}{k}tan\left(\theta_{2}\left(12,\frac{12}{k}\alpha\right)\right).
\end{equation}

Now, $\alpha=\sqrt{k-1}$, then $\frac{12}{k}\alpha\leq \sqrt{13}$, hence
\begin{equation}
tan\left(\theta_{2}(k,\alpha)\right)< \frac{12}{k}tan\left(\theta_{2}(12,\sqrt{13})\right)<\frac{2}{k}.
\end{equation}

The normal radius is $arccos(1-\frac{dim M_{1}}{k-1})$, we compare $2arctan\frac{2}{k}$ and $arccos(1-\frac{dim M_{1}}{k-1})$ as follows:

Note $2arctan\frac{2}{k}<\frac{4}{k}$, $1-cos\frac{4}{k}<\frac{8}{k^2}$ and $dim M_{1}\geq 4$, then the proof is followed by the above relations.
\end{proof}

\begin{remark}
The minimal product is suitable for various mixed cases which the factors could be any concrete embedded minimal submanifolds in spheres or the spheres themselves. For the cases of products of Grassmannians, there is a very challenging question here on the complete classification of the associated cones, we list some other cones which deserve further considerations:

~(1)Whether the cones of dimension less than $7$ are area-minimizing, stable or unstable, similar to the cones over products of spheres? Except the products of spheres(i.e. all the factors are belong to $\mathbb{R}P^{1}, \mathbb{C}P^{1},\mathbb{H}P^{1}$), all the others cannot be proved area-minimizing directly by Curvature Criterion in \cite{lawlor1991sufficient}, such examples are: $S^{1}\times \mathbb{R}P^2$, $S^1 \times G(2,4;\mathbb{R})$, the product of two Veronese maps $\mathbb{R}P^2 \times \mathbb{R}P^2$, and $\mathbb{R}P^2 \times \mathbb{R}P^3$, etc.

~(2)Cones of dimension $7$ which have reduced factors(not all), and the minimum polynomials are $F(t)$, could these cones be area-minimizing?

~(3)There are some cones of dimension bigger than $7$ which it owns mixed factors which could be $G(n,m;\mathbb{F})$, $\mathbb{O}P^2$, or $Pl\ddot{u}cker$ embedded oriented real Grassmannians, etc.
\end{remark}

\vspace{0.3cm}
\noindent\textbf{Acknowledgments}. This work is supported by NSFC No.11871450. We would like to thank Professor YongSheng Zhang for valuable suggestions.

\end{document}